\documentclass[12pt,journal]{IEEEtran}

\usepackage[hmargin=1.2in,vmargin=1.2in]{geometry}
\usepackage{amsmath,amsthm,amssymb,amsfonts,verbatim}
\usepackage{graphicx}
\usepackage[colorlinks, linkcolor=blue,  anchorcolor=blue,
            citecolor=blue
            ]{hyperref}
\usepackage[hmargin=1.2in,vmargin=1.2in]{geometry}
\usepackage{booktabs}
\usepackage{bm}

\title{A new family of binary sequences with a low correlation via elliptic curves}
\author{Lingfei Jin, Liming Ma, Chaoping Xing and Runtian Zhu
\thanks{L. Jin and R. Zhu are  with the Shanghai Key Laboratory of Intelligent Information Processing, School of Computer Science, Fudan University, Shanghai 200433, China, and with the State Key Laboratory of Cryptology, P. O. Box, 5159, Beijing 100878, China. (emails: lfjin@fudan.edu.cn, 23110240100@m.fudan.edu.cn).}
\thanks{L. Ma is with the School of Mathematical Sciences, University of Science and Technology of China, Hefei 230026, China (e-mail: lmma20@ustc.edu.cn).}
\thanks{C. Xing is with School of Electronic Information and Electrical Engineering, Shanghai Jiao Tong University, Shanghai 200240, China (email: xingcp@sjtu.edu.cn).}

}

\newtheorem{lemma}{Lemma}[section]
\newtheorem{theorem}[lemma]{Theorem}
\newtheorem{cor}[lemma]{Corollary}

\newtheorem{prop}[lemma]{Proposition}

\newtheorem{rem}[lemma]{Remark}

\theoremstyle{remark}

\renewcommand{\epsilon}{\varepsilon}
\renewcommand{\le}{\leqslant}
\renewcommand{\ge}{\geqslant}

\def\Gal{{\rm Gal}}

\newcommand{\vnote}[1]{}

\def\F{\mathbb{F}}
\def\Z{\mathbb{Z}}

\def \mL {\mathcal{L}}

\def \mS {\mathcal{S}}

\newcommand{\Ga}{\alpha}

\newcommand{\Gs}{\sigma}

\newcommand{\s}{\sigma}

\DeclareMathOperator{\Tr}{Tr}

\def \bs {{\mathbf s}}

\def \bo {{\mathbf 0}}

\def\mS{{\mathcal S}}

\DeclareMathOperator{\Aut}{{Aut}}

\def\Gal{{\rm Gal}}

\newcommand \mcal \mathcal

\onecolumn

\begin{document}
\maketitle

\begin{abstract}
In the realm of modern digital communication, cryptography, and signal processing, binary sequences with a low correlation properties play a pivotal role. In the literature, considerable efforts have been dedicated to constructing good binary sequences of various lengths. As a consequence, numerous constructions of good binary sequences have been put forward. However, the majority of known constructions leverage the multiplicative cyclic group structure of finite fields $\mathbb{F}_{p^n}$, where $p$ is a prime and $n$ is a positive integer.
Recently, the authors made use of the cyclic group structure of all rational places of the rational function field over the finite field $\F_{p^n}$,  and firstly constructed good binary sequences of length $p^n+1$ via cyclotomic function fields over $\mathbb{F}_{p^n}$ for any prime $p$ \cite{HJMX24,JMX22}. This approach has paved a new way for constructing good binary sequences. Motivated by the above constructions, we exploit the cyclic group structure on rational points of elliptic curves to design a family of binary sequences of length $2^n+1+t$ with a low correlation for many given integers $|t|\le 2^{(n+2)/2}$. Specifically, for any positive integer $d$ with $\gcd(d,2^n+1+t)=1$, we introduce a novel family of binary sequences of length $2^n+1+t$, size $q^{d-1}-1$, correlation bounded by $(2d+1) \cdot 2^{(n+2)/2}+ |t|$, and a large linear complexity via elliptic curves.
\end{abstract}


\section{Introduction}
Sequences with a low correlation have found  essential applications in various domains, including code-division multiple access (CDMA), spread spectrum systems, broadband satellite communications and cryptography \cite{AP07, CHPW15, K10,16ZXE}. These applications demand sequences with favorable characteristics, such as a low correlation (autocorrelation and cross-correlation) and a large family size. Since the later 1960s, many families of binary sequences with good properties have been constructed via various methods. 

Among them, the well-known Gold sequence family is the oldest binary sequence of length $2^n-1$, size $2^n+1$ 
having three-level out-of-phase auto- and cross-correlation for odd integer $n$.
The family of Gold sequences  can be obtained by XOR-ing two m-sequences \cite{m} of the same length with each other \cite{Span}.
Later, Kasami designed a sequence family using m-sequences and their decimations \cite{Kas}, known as the Kasami sequences.
Besides, there are some known families of sequences of length  $2^n-1$ with good correlation properties, such as  bent function sequences \cite{Bent}, No sequences \cite{No}, and Trace Norm sequences \cite{Klap}.
In 2011, Zhou and Tang generalized the modified Gold sequences and obtained a family of binary sequences with length $2^n-1$ for $n=2m+1$, size $2^{\ell n}+\cdots+2^n-1$, and correlation $2^{m+\ell}+1$ for each $1\le \ell \le m$ \cite{ZT}.
In addition to the above mentioned sequences of length $2^n-1$, there is some work devoted to the construction of good binary sequences with length of forms such as $2(2^n-1)$ and $(2^n-1)^2$ \cite{Uda,Tang3,Tang2,gong}.
In \cite{Uda}, Udaya and Siddiqi obtained a family of $2^{n-1}$ binary sequences, length $2(2^n-1)$  satisfying the Welch bound for odd $n$.
This family was extended to a larger one with  $2^n$ sequences and the same correlation in \cite{Tang3}. In \cite{Tang2}, the authors presented two optimal families of sequences of length $2^n-1$ and $2(2^n-1)$ for odd integer $n$.
In 2002, Gong presented a family of binary sequences with size $2^n-1$, length $(2^n-1)^2$, and correlation $3+2(2^n-1)$ \cite{gong}.
These constructions mainly made use of the cyclic structure of multiplicative group $\F_{2^n}^*$.
In order to design sequences of length $2^n+1$, the cyclic structure of all rational points of projective lines over $\F_{2^n}$ in \cite{JMX20} is employed to construct a family of binary sequence of length $2^n+1$, size $2^n-1$ and correlation bounded by $2^{(n+2)/2}$ via cyclotomic function fields \cite{JMX22}. 
Using the cyclic structure of rational points of elliptic curves, Xing {\it et. al.} provided a family of binary sequences with a low correlation and a large linear span in \cite{XKD03}.

There are also some constructions of binary sequences which are based on finite fields of odd characteristics \cite{Lempel,Legendre,Pat,Rushanan,Su, Jin21}.
Paterson proposed a family of pseudorandom binary sequences based on Hadamard difference sets and maximum distance separable (MDS) codes for length $p^2$ where $p\equiv 3\ (\text{mod } 4)$ is a prime \cite{Pat}.
In 2006, a family of binary sequences of length $p$ which is an odd prime, family size $(p-1)/2$, and correlation bounded by $5+2\sqrt{p}$ was given in \cite{Rushanan}, now known as Weil sequences.
The idea of constructing Weil sequences is to derive sequences from single quadratic residue based on Legendre sequences using a shift-and-add construction.
In 2010, by combining the $p$-periodic Legendre sequence and the $(q-1)$-periodic Sidelnikov sequence, Su {\it et al.} introduced a new sequence of length $p(q-1)$, called the Lengdre-Sidelnikov sequence \cite{Su}.
In 2021, Jin {\it et al.} provided several new families of sequences with low correlations of length $q-1$ for any prime power $q$ by using the cyclic multiplication group $\F_q^*$ and sequences of  length $p$ for any odd prime $p$  from the cyclic additive group $\F_p$ in \cite{Jin21}.   The idea of constructing  binary sequences is  using  multiplicative quadratic characters over finite fields of odd characteristic. By exploiting the cyclic group structure of rational places of the rational function field $\F_{p^n} (x)$,  a family of binary sequence of length $p^n + 1$, size  $p^n -2$,  and correlation bounded by $4+2p^{n/2}$ was constructed  via automorphisms of cyclotomic function fields with odd characteristics in \cite{HJMX24}. 

While several constructions of binary sequences with favorable parameters have been introduced, the available choices for sequence lengths remain limited. For applications in different scenarios, adding or deleting sequence values usually destroys the good correlation properties. 
Consequently, there is a pressing need for constructing binary sequences that exhibit low correlation while offering greater flexibility in parameter selection. However, identifying such binary sequences with novel parameters remains an open and challenging problem in the field.

%\begin{figure*}[htbp]
\begin{table}[h]\label{tab:summary}
	\setlength{\abovecaptionskip}{0pt}%
	\setlength{\belowcaptionskip}{10pt}%
	\caption{PARAMETERS OF SEQUENCE FAMILIES}
	\centering
{
	\begin{tabular}{@{}cccc@{}}
		\toprule
		Sequence                             & Length $N$                                & Family Size                           &  Bound of Correlation                       \\ \midrule
		\multicolumn{1}{|c|}{Gold(odd) \cite{Span}}      & \multicolumn{1}{c|}{$2^n-1$, $n$ odd}                   & \multicolumn{1}{c|}{$N+2$}             & \multicolumn{1}{c|}{$1+\sqrt{2}\sqrt{N+1}$} \\ \midrule
		\multicolumn{1}{|c|}{Gold(even) \cite{Span}}     & \multicolumn{1}{c|}{$2^n-1,n=4k+2$}            & \multicolumn{1}{c|}{$N+2$}             & \multicolumn{1}{c|}{$1+2\sqrt{N+1}$}        \\ \midrule
		\multicolumn{1}{|c|}{Kasami(small) \cite{Kas}}  & \multicolumn{1}{c|}{$2^n-1, n \text{ even}$}            & \multicolumn{1}{c|}{$\sqrt{N+1}$}      & \multicolumn{1}{c|}{$1+\sqrt{N+1}$}         \\ \midrule
		\multicolumn{1}{|c|}{Kasami(large) \cite{Kas}}  & \multicolumn{1}{c|}{$2^n-1,n=4k+2$}            & \multicolumn{1}{c|}{$(N+2)\sqrt{N+1}$} & \multicolumn{1}{c|}{$1+2\sqrt{N+1}$}        \\ \midrule
		\multicolumn{1}{|c|}{Bent \cite{Bent}}           & \multicolumn{1}{c|}{$2^n-1, n=4k$}                & \multicolumn{1}{c|}{$\sqrt{N+1}$}         & \multicolumn{1}{c|}{$1+\sqrt{N+1}$}          \\ \midrule
		\multicolumn{1}{|c|}{No \cite{No}}           & \multicolumn{1}{c|}{$2^n-1, n \text{ even}$}                & \multicolumn{1}{c|}{$\sqrt{N+1}$}         & \multicolumn{1}{c|}{$1+\sqrt{N+1}$}          \\ \midrule
		\multicolumn{1}{|c|}{Trace Norm \cite{Klap}}           & \multicolumn{1}{c|}{$2^n-1, n \text{ even}$}                & \multicolumn{1}{c|}{$\sqrt{N+1}$}         & \multicolumn{1}{c|}{$1+\sqrt{N+1}$}          \\ \midrule
		\multicolumn{1}{|c|}{Tang et al. \cite{Tang3}}  & \multicolumn{1}{c|}{$2(2^n-1),n$ odd }            & \multicolumn{1}{c|}{$\frac{N}{2}+1$} & \multicolumn{1}{c|}{$2+\sqrt{N+2}$}        \\ \midrule
		\multicolumn{1}{|c|}{Gong \cite{gong}}  & \multicolumn{1}{c|}{$(2^n-1)^2$, $2^n-1$ prime}            & \multicolumn{1}{c|}{$\sqrt{N}$ }& \multicolumn{1}{c|}{$3+2\sqrt{N}$}        \\ \midrule
	%\multicolumn{1}{|c|}{JMX \cite{JMX22}} &\multicolumn{1}{c|}{{\bf $2^n+1$}}  &\multicolumn{1}{c|}{{\bf $N-2$}}         & \multicolumn{1}{c|}{\bf {$2\sqrt{N-1}$}}          \\

		\multicolumn{1}{|c|}{Paterson \cite{Pat}}       & \multicolumn{1}{c|}{$p^2$, $p$ prime, $p\equiv 3(\text{mod } 4)$}    & \multicolumn{1}{c|}{$N$}               & \multicolumn{1}{c|}{$5+4\sqrt{N}$}          \\ \midrule
		\multicolumn{1}{|c|}{Paterson \cite{Pat}}  & \multicolumn{1}{c|}{$p^2$, $p$ prime, $p\equiv 3(\text{mod } 4)$}    & \multicolumn{1}{c|}{$\sqrt{N}+1$}      & \multicolumn{1}{c|}{$3+2\sqrt{N}$}          \\ \midrule
		\multicolumn{1}{|c|}{Weil \cite{Rushanan}}           & \multicolumn{1}{c|}{$p$, odd prime}                & \multicolumn{1}{c|}{$(N-1)/2$}         & \multicolumn{1}{c|}{$5+2\sqrt{N}$}          \\ \midrule

		\multicolumn{1}{|c|}{Jin et al. \cite{Jin21}} & \multicolumn{1}{c|}{$p-1$, $p\ge17$ odd prime power} & \multicolumn{1}{c|}{$N+3$}               & \multicolumn{1}{c|}{$6+2\sqrt{N+1}$}          \\ \midrule
		\multicolumn{1}{|c|}{Jin et al. \cite{Jin21}} & \multicolumn{1}{c|}{$p-1$, $p\ge11$ odd prime power} & \multicolumn{1}{c|}{$N/2$}         & \multicolumn{1}{c|}{$2+2\sqrt{N+1}$}          \\ \midrule
		\multicolumn{1}{|c|}{Jin et al. \cite{Jin21}} & \multicolumn{1}{c|}{$p$, $p\ge 17$ odd prime} & \multicolumn{1}{c|}{$N$}               & \multicolumn{1}{c|}{$5+2\sqrt{N}$}          \\ \midrule
		\multicolumn{1}{|c|}{Jin et al. \cite{Jin21}} & \multicolumn{1}{c|}{$p$, $p \ge 11$ odd prime} & \multicolumn{1}{c|}{$(N-1)/2$}         & \multicolumn{1}{c|}{$1+2\sqrt{N}$}          \\ \midrule
			\multicolumn{1}{|c|}{{HJMX \cite{HJMX24}}} & \multicolumn{1}{c|}{ {$p^n+1$, $p$ odd prime}} & \multicolumn{1}{c|}{{ $N-3$}}         & \multicolumn{1}{c|}{ {$4+2\sqrt{N-1}$}}          \\
			\midrule

				\multicolumn{1}{|c|}{{JMX \cite{JMX22}}} & \multicolumn{1}{c|}{ {$2^n+1$}} & \multicolumn{1}{c|}{{ $N-2$}}         & \multicolumn{1}{c|}{ {$2\sqrt{N-1}$}}          \\
			\midrule			
				\multicolumn{1}{|c|}{{XKD \cite{XKD03} }} & \multicolumn{1}{c|}{ {$2^n$}} & \multicolumn{1}{c|}{{ $(N^2 -2N + 2)/3$}}         & \multicolumn{1}{c|}{ {$1 +14\sqrt {N} $}}          \\
	\midrule
	\multicolumn{1}{|c|}{{XKD \cite{XKD03} }} & \multicolumn{1}{c|}{ {$q + 1 \pm \sqrt q, q = 4^n$}} & \multicolumn{1}{c|}{{ $(q - \sqrt q)/2$}}         & \multicolumn{1}{c|}{ {$11\sqrt {q} $}}          \\
	\midrule 
	\multicolumn{1}{|c|}{{XKD \cite{XKD03} }} & \multicolumn{1}{c|}{ {$q + 1\pm \sqrt {2q} , q = 2 \times 4^n$}} & \multicolumn{1}{c|}{{ $(q - \sqrt {2q})/2$}}         & \multicolumn{1}{c|}{ {$(10+\sqrt 2)\sqrt {q} $}}          \\
	\midrule 
	%			\multicolumn{1}{|c|}{{\bfseries Our construction}} & \multicolumn{1}{c|}{ {$\bm {2^n}$}} & \multicolumn{1}{c|}{{ $\bm{N-1}$}}         & \multicolumn{1}{c|}{ {$\bm{14 \sqrt N + 1}$}}          \\
	%	\midrule                                                                                                                                                                  
	\multicolumn{1}{|c|}{{\bfseries Our construction}} & \multicolumn{1}{c|}{ {$  {2^n}$}} & \multicolumn{1}{c|}{{ $ {N^2 - 1}$}}         & \multicolumn{1}{c|}{ {$ {1 + 14 \sqrt N }$}}          \\
	\midrule
	\multicolumn{1}{|c|}{{\bfseries Our construction}} & \multicolumn{1}{c|}{ {$  {q + 1 \pm \sqrt q, q = 4^n}$}} & \multicolumn{1}{c|}{{ $ {q - 1}$}}         & \multicolumn{1}{c|}{ {$ {11 \sqrt q }$}}          \\
	\midrule
	\multicolumn{1}{|c|}{{\bfseries Our construction}} & \multicolumn{1}{c|}{ {$  {q + 1 \pm \sqrt {2q}, q = 2 \times 4^n }$}} & \multicolumn{1}{c|}{{ $ {q - 1}$}}         & \multicolumn{1}{c|}{ {$ {(10+\sqrt 2)\sqrt q }$}}          \\
\bottomrule
	\end{tabular}
}
\end{table}
%\end{figure*}

\subsection{Our main result and techniques}
In this paper, we provide an explicit construction of binary sequences of length $2^n+1+t$ with a low correlation via cyclic elliptic function fields over the finite field $ \F_{2^n}$. 
Let $d\ge2 $ be a positive integer with $\gcd(d, 2^n+1+t)=1$. We establish a family of binary sequences of length $2^n+1+t$, size $q^{d-1}-1$ and correlation upper bounded by  $2(2d+1)\sqrt{q}+1$. 
For the purpose of comparison, we provide a summary of various families of binary sequences with a low correlation in Table \ref{tab:summary}.

In order to understand the main idea of this article better, we give a high-level description of our techniques.
We use language of curves and function fields interchangeably. 
To produce a sequence of length $q +1+t$, we use the additive cyclic group structures under the group law $\oplus$ of rational places of cyclic elliptic curves over a finite field $\F_{q}$ \cite{Ru87}. 
The automorphism group of an elliptic curve $E$ over $\F_q$ is in the form $T_E \rtimes \Aut (E,O)$, where $\Aut (E,O)$ is the automorphisms of $E$ fixing the infinity place $O$, and  $T_E$ corresponds to translation group of $E$ \cite{MX23}. 
In particular, $T_E$ is a cyclic group which is generated by a rational place $P$ in this case, and there exists a generator $\sigma =\tau_P\in T_E$ that makes the set $\{ \sigma ^j (P) \}_{j = 0}^{q+t}$ being exactly the set of rational places of $E$. 
Our sequences are constructed via the absolute trace of evaluations of rational functions on such cyclically ordered rational places $\{\Gs^j(P)\}_{j=0}^{q+t}$.
To obtain a lower correlation, we need to choose rational functions $z_1,z_2,\cdots,z_\ell$ for some positive integer $\ell$ carefully.
The cyclic group structure of $T_E$ allows us to convert the process of determining the correlation of this family of binary sequences to the estimation of number of rational places on the Artin-Schreier extensions $E_j = E(y)$ where $y^2+y = z_j + \sigma ^{-u} (z_j)$ with $1 \le j \le \ell$, $1 \le u \le q+t$ and Artin-Schreier extensions $E_{j,k} = E(y)$ where $y^2 +y= z_j + \sigma ^{-u} (z_k)$ with $1 \le j \neq k \le \ell$, $0 \le u \le q+t$. 
In order to obtain Artin-Schreier curves with small genus, first we need to choose function $z_i$ such that $z_i$ and $\Gs^{-t}(z_j)$ have less poles. Hence, we choose $z_i$ from a Riemann-Roch space associated to a divisor which is just a place $Q$ with degree $d\ge 2$ of $E$.
In order to obtain Artin-Schreier curves with less rational places, $z_j + \sigma ^{-t} (z_k)$ cannot be constant, since constant field extensions may produce more rational points. 
Thus, $z_j$ can be  chosen as nonzero elements from the complementary dual space of $\F_q$ in the Riemann-Roch space $\mL(Q)$.

\subsection{Organization of this paper}
In Section \ref{sec:2}, we provide preliminaries on binary sequences, extension theory of function fields, elliptic function fields and Artin-Schreier extensions. 
In Section \ref{sec:3}, we present a theoretical construction of binary sequences with a low correlation via elliptic function fields over finite fields with characteristic two. 
In Section \ref{sec:4}, we give an algorithm to generate such a family of binary sequences and provide many numerical results with the help of the Sage software.

\section{Preliminaries}\label{sec:2}
In this section, we present some preliminaries on binary sequences, extension theory of function fields, elliptic function fields and Artin-Schreier extensions.

\subsection{Binary sequences}
For a positive integer $N$, a binary sequence of length $N$ is a vector $\mathbf s \in \F_2^N$, where $\F_2$ is the finite field of two elements. 
Let $\mathcal{S}$ be a family of binary sequences in $\F_2^N$. The number of binary sequences in $\mathcal{S}$ is called the family size of $\mathcal{S}$. 
For a binary sequence $\mathbf s = (s_0, s_1, \dots, s_{N-1})$, the autocorrelation $A_t(\mathbf s)$ of $\mathbf s$ at delay $t$ for $1 \le t \le N-1$ is defined as
\[A_t(\mathbf s):= \sum_{j=0}^{N-1} (-1)^{s_j+ s_{j + t}}, \]
where $j + t$ is identified with the least non-negative integer after taking modulo $N$. 
For two distinct binary sequences $\mathbf u = (u_0, u_1$, $\dots, u_{N-1})$ and $\mathbf v = (v _{0}, v_1 , \dots, v_{N-1})$, 
their cross-correlation at delay $t$ with $0 \le t \le N-1$ is defined as 
\[C_t(\mathbf u, \mathbf v) := \sum_{j =0}^{N-1} (-1)^{u_j + v_{j+t}}. \]
For the family of binary sequences  $\mcal S$, its correlation is defined by  
\[\operatorname{Cor}(\mcal S):= \max \left\{\max_{\substack{ \mathbf s \in \mcal S \\ 1 \le t \le N-1} } \{
|A_t(\mathbf s)|\}, \max _{\substack{ \mathbf u\neq \mathbf v \in \mcal S \\ 0 \le t \le N-1}} \{| C_t (\mathbf u, \mathbf v) |\}\right\}.\]
	
In addition, we can define the linear complexity of binary sequences as follows. For any sequence $\mathbf 0 \neq \mathbf s \in \mcal{S}$, the linear complexity of $\mathbf{s}$, denoted by $\ell (\mathbf s)$, is the smallest positive integer $\ell$ such that there exists $\ell + 1$ elements $\lambda_0, \lambda_1, \dots, \lambda_\ell \in \F_2$ with $\lambda_0 = \lambda_ \ell = 1$ satisfying 
$$\sum_{i=0}^\ell \lambda_i s_{i+u}=0$$
for all $0\le u\le N-1$. Define the linear complexity of the family of sequences $\mcal{S}$ by
$$\operatorname{LC}(\mcal S):=\min\{\ell(\bf{s}): \bf{s}\in \mcal{S}\}.$$

\subsection{Extension theory of function fields}
Let $F/\F_q$ be an algebraic function field with genus $g$ over the full constant field $\F_q$. Let $\mathbb P_F$ denote the set of places of $F$. Any place with degree one of $F$ is called rational, and all rational places of $F$ form a set denoted by $\mathbb P_F^1$. From the Serre bound \cite[Theorem 5.3.1]{St09}, the number $N(F)$ of rational places of $F$ is bounded by $$|N(F)-q-1|\le g\cdot \lfloor 2\sqrt{q}\rfloor. $$
Here $\lfloor x \rfloor$ stands for the integer part of $x\in \mathbb{R}$.

Let $\nu_P$ be the normalized discrete valuation of $F$ with respect to the place $P\in \mathbb{P}_F$. The principal divisor of a nonzero element $z\in F$ is given by $(z):=\sum_{P\in \mathbb{P}_F}\nu_P(z)P$.
For a divisor $G$ of $F/\F_q$, the Riemann-Roch space associated to $G$ is defined by
\[\mL(G):=\{z\in F^*:\; (z)+G\ge 0\}\cup\{0\}.\]
If $\deg(G)\ge 2g-1$, then $\mL(G)$ is a vector space over $\F_q$ of dimension $\deg(G)-g+1$ from the Riemann-Roch theorem \cite[Theorem 1.5.17]{St09}.

Let $\Aut(F/\F_q)$ denote the automorphism group of $F$ over $\F_q$, i.e.,
$$\Aut(F/\F_q)=\{\Gs: F\rightarrow F \;|\; \Gs  \mbox{ is an } \mbox{automorphism of } F \text{ and } \sigma|_{\F_q} = \mathrm {id}\}.$$
Let us consider the group action of automorphism group $\Aut(F/\F_q)$ on the set of places of $F$.
For any place $P$ of $F$, $\sigma(P)$ is a place of $F$ for any automorphism $\sigma\in \Aut(F/\F_q)$. 
From \cite[Lemma 1]{NX14}, we have the following results.

\begin{lemma}\label{lem:2.1}
For any automorphism $\s\in \Aut(F/\F_q)$, $P\in \mathbb{P}_F$ and $f\in F$, we have
\begin{itemize}
\item[(1)] $\deg(\s(P))=\deg(P)$;
\item[(2)] $\nu_{\s(P)}(\s(f))=\nu_P(f)$;
\item[(3)] $\s(f)(\s(P))=f(P)$, provided that $\nu_P(f)\ge 0$.
\end{itemize}
\end{lemma}

Let $E/\F_q$ be a finite extension of $F/\F_q$. The Hurwitz genus formula \cite[Theorem 3.4.13]{St09} yields
$$2{g(E)}-2=[E:F]\cdot (2g(F)-2)+\deg \operatorname{Diff}(E/F),$$
where $\operatorname{Diff}(E/F)$ stands for the different of $E/F$.
For any $P\in \mathbb{P}_F$ and $Q\in \mathbb{P}_F$ with $Q|P$, let $d(Q|P), e(Q|P)$ and $f(Q|P)$ be the different exponent, ramification index and relative degree of $Q|P$, respectively. The different of $E/F$ is defined as 
$\text{Diff}(E/F)=\sum_{Q\in \mathbb{P}_F} d(Q|P) Q$ and different exponents can be determined by the Dedekind different theorem and Hilbert's different theorem \cite[Theorem 3.5.1 and Theorem 3.8.7]{St09}.
In particular, the decomposition group of $Q$ over $P$ is defined as $G_Z(Q|P)=\{ \sigma \in \Gal(E/F) \colon \sigma (Q) =Q\}$ and its order is $e(Q|P)f(Q|P)$.

\subsection{Elliptic function fields}
An elliptic function field $E$ over $\F_q$ is an algebraic function field over $\F_q$ of genus one and a rational place $O$ which is called the infinite place of $E$. 
The divisor group Div$(E)$ of $E/\F_q$ is defined as the free abelian group generated by $\mathbb{P}_E$. The set of divisors of degree zero forms a subgroup of Div$(E)$, denoted by $\text{Div}^0(E)$.
Two divisors $A$ and $B$ in $\text{Div}(E)$ are called equivalent if there exist $z\in E^*$ such that $A=B+(z)$, which is denote by $A\sim B$.
The set of principal divisors $\text{Princ}(E)=\{(x)=\sum_{P\in\mathbb{P}_E}\nu_P(x) P: x\in E^*\}$ is called the group of principal divisors of $E/\F_q$. 
The factor group $\text{Cl}^0(E)=\text{Div}^0(E)/\text{Princ}(E)$ is called the group of divisor classes of degree zero of $E/\F_q$.
There is a bijective map  
	\[
		\Phi \colon \left\{
		\begin{array}{rcl}
			\mathbb P_E^1 & \longrightarrow & \operatorname{Cl}^0 (E) \\
			P & \longmapsto & [P- O]
		\end{array}
		\right.
	\]
from the set of rational places of $E$ and the group of divisor classes of degree zero, 
where $[\cdot]$ denotes the equivalence class of $\cdot$ in the quotient group $\operatorname{Cl}^0 (E)$. 
Hence, the set of rational places $\mathbb P^1_E$ can be endowed with a group operation $\oplus$ as \[P \oplus Q = \Phi ^{-1} (\Phi (P) + \Phi (Q)).\]
It is easy to see that $P \oplus Q = R$ if and only if $P + Q \sim R+ O$. Under this operation, $O$ is the identity of the additive group $\mathbb P_E^1 $. For convenience, we can write $[2]P = P \oplus P$ and inductively define $[m+1]P = [m]P \oplus P$ for any integer $m\in \mathbb{Z}$. 

\begin{lemma}\label{lem:2.2}
Let $E/\F_q$ be an elliptic function field and let $P,Q$ be two rational places of $E$. Then
$P-Q=(z)$ for some element $z\in E^*$ if and only if $P=Q$. 
\end{lemma}
	
For elliptic curves, we can consider the isogenies between them, which are non-constant smooth morphisms that preserve the infinity point. It is well known that elliptic curves over $\F_q$ are isogenous if and only if they have the same number of $\F_q$-rational points  \cite{Wa69}.
	
\begin{lemma} \label{lem:2.3}
The isogeny classes of elliptic curves over $\F_q$ for $q = p^n$ are in one-to-one correspondence with the rational integers $t$ whose absolute value not exceed $2 \sqrt q$ and satisfying some one of the following conditions:
	\begin{enumerate}
			\item[(i)]  $\gcd (t, p ) = 1$;
			\item[(ii)]  $t =\pm 2 \sqrt q$ if $n$ is even; \label{item:special}
			\item[(iii)]  $t = \pm \sqrt q$ if $n$ is even and $p \not\equiv 1 \pmod 3$;
			\item[(iv)]  $t = \pm p^{(n+1)/2}$ if $n$ is odd and $p = 2$ or $3$;
			\item[(v)]  $t = 0$ if either (i) $n$ is odd or (ii) $n$ is even and $p \not \equiv 1 \pmod 4$. 
	\end{enumerate}
Furthermore, an elliptic curve in the isogeny class corresponding to $t$ has $q + 1 + t$ rational points. 
\end{lemma} 

The additive group structure of rational places $\mathbb P^1_E$  of any elliptic function field $E/\F_q$ can be found from \cite[Theorem 3]{Ru87}.

\begin{lemma}\label{lem:2.4}
Let $\F_q$ be the finite field with $q=p^n$ elements. Let $h=\prod_\ell \ell^{h_\ell}$ be a possible number of rational places of an elliptic function field $E$ of $\F_q$. Then all the possible groups $\mathbb{P}_E^1$ are the following
\[\mathbb{Z}/p^{h_p}\mathbb{Z} \times \prod_{\ell \neq p}\left( \mathbb{Z}/\ell^{a_\ell}\mathbb{Z}\times \mathbb{Z}/\ell^{h_\ell-a_\ell}\mathbb{Z}\right)\]
with \begin{itemize}
 \item[(a)] in case (ii) of Lemma \ref{lem:2.3}: each $a_\ell$ is equal to $h_\ell/2$, i.e, $\mathbb{P}_E^1\cong \mathbb{Z}/(\sqrt{q}\pm1)\mathbb{Z} \times  \mathbb{Z}/(\sqrt{q}\pm1)\mathbb{Z}.$
\item[(b)] in other cases of Lemma \ref{lem:2.3}: $a_\ell$ is an arbitrary integer satisfying $0\le a_\ell \le \min\{\nu_\ell(q-1),[h_\ell/2]\}$.
In cases (iii) and (iv) of Lemma \ref{lem:2.3}: $\mathbb{P}_E^1\cong \mathbb{Z}/h\mathbb{Z}.$
 In cases (v) of Lemma \ref{lem:2.3}: if $q\not\equiv -1(\text{mod } 4)$, then $\mathbb{P}_E^1\cong \mathbb{Z}/(q+1)\mathbb{Z}$; otherwise, $\mathbb{P}_E^1\cong \mathbb{Z}/(q+1)\mathbb{Z}$ or $\mathbb{P}_E^1\cong \mathbb{Z}/2\mathbb{Z} \times \mathbb{Z}/\frac{q+1}{2}\mathbb{Z}$.
\end{itemize}
\end{lemma}
	
If all rational places of an elliptic function field $E/\F_q$ form a cyclic group under the addition $\oplus$, then $E/\F_q$ is called a cyclic elliptic function field. 
In particular, we will employ this cyclic structure of the rational places of elliptic function fields with even characteristic to construct binary sequences as \cite[Lemma 5.3]{XKD03}. 

\begin{lemma}\label{lem:2.5}
		Let $q=2^n$. Let $t$ be an integer satisfying one of the following conditions:
		\begin{itemize}
			\item[(1)] $t$ is an odd integer between $-2\sqrt{q}$ and $2\sqrt{q}$;
			\item[(2)]  $t=0$;
			\item[(3)]  $t=\pm \sqrt{q}$ if $n$ is even, and $t=\pm \sqrt{2q}$ if $n$ is odd.
		\end{itemize}
Then there exists a cyclic elliptic function field $E$ over $\F_q$ has $q+1+t$ rational places. 
\end{lemma}

\subsection{Artin-Schreier extensions}
In this subsection, we summarize the main properties of Artin-Schreier extensions from \cite[Theorem 3.1.9]{NX01}.

\begin{prop}\label{prop:2.6}
Let $q=2^n$ be a prime power. 
Let $F/\F_q$ be the function field over $\F_q$ with genus $g$. Let $f$ be an element in $F$ such that $f\neq z^2+z$ for all $z\in F$.
Let \[F^\prime=F(y) \text{ with } y^2+y=f.\] For any place $P\in \mathbb{P}_F$, we define the integer $m_P$ by
$$
m_P=\begin{cases}
\ell &\text{ if } \exists  z\in F \text{ satisfying }  \nu_P(u-(z^2+z))=-\ell \\ & \text{ with } \ell\in \mathbb{Z}_+ \text{ and } \ell \not \equiv 0 (\text{mod }2),\\
-1 & \text{ if } \nu_P(u-(z^2+z))\ge 0 \text{ for some } z\in F.
\end{cases}$$
Then we have the following results. 
\begin{itemize}
\item[(a)] $F^\prime/F$ is a cyclic Galois extension of degree two. The automorphisms of $F^\prime/F$ are given by $\sigma(y)=y+u$ with $u=0,1.$
\item[(b)] $P$ is unramified in $F^\prime/F$ if and only if $m_P=-1$.
\item[(c)] $P$ is totally ramified in $F^\prime/F$  if and only if $m_P>0$. Denote by $P^\prime$ the unique place of $F^\prime$ lying over $P$. Then the different exponent $d(P^\prime|P)$ is given by \[d(P^\prime|P)=m_P+1.\]
\item[(d)] Assume that there is at least one place $Q$ satisfying $m_Q>0$, then the full constant field of $F^\prime$ is $\F_q$ and the genus of $F^\prime$ is
\[g(F^\prime)=2g-1+\frac{1}{2}\sum_{P\in \mathbb{P}_{F}} (m_P+1)\cdot \deg(P).\]
\end{itemize}
\end{prop}

\section{Binary sequences with a low correlation via elliptic curves}\label{sec:3}
Motivated by the construction of families of binary sequences with a low correlation via automorphism groups of cyclotomic function fields in \cite{HJMX24,JMX22}, 
we provide a new construction of binary sequences via elliptic curves.
Compared with the result in \cite{XKD03}, our construction achieves the same low correlation and high linear complexity but offers a larger family size.

\subsection{A new construction of binary sequences via elliptic curves}
Let $n$ be a positive integer and $q = 2^n$. Let $ \F_q$ be the finite field with $q$ elements. 
In this subsection, we shall employ the cyclic group structure of rational points and automorphism groups of elliptic curves over $\F_{2^n}$ to construct a new family of binary sequences with a low correlation and a large linear complexity. 
	
Let $E/\F_q$ be an elliptic function field defined over $\F_q$ with $q+1+t$ rational places. 
Let $\tau_P$ be the automorphism of $E$ induced by $\tau_P(Q)=P\oplus Q$ for any rational place $P$ and $Q$ of $E$. 
Let  $T_E=\{\tau_P: P\in \mathbb{P}^1_E\}$  be the translation group of $E$. Let $O$ be the infinity place of $O$ and $\Aut(E,O)$ be the automorphism group of $E$ fixing the infinity place $O$. 
From \cite[Theorem 3.1]{MX23}, the automorphism group of $E/\F_q$ is $\Aut(E/\F_q)\cong T_E \rtimes \Aut(E,O)$. 
Let $F$ be the fixed subfield of $E$ with respect to the translation group $T_E$, i.e.,
$$F=E^{T_E}=\{z\in E: \sigma(z)=z \text{ for any } \sigma\in T_E\}.$$
From Galois theory, $E/F$ is a Galois extension with Galois group $\Gal(E/F)\cong T_E$. 
Since the number of rational places of $E$ is $q+1+t$, there exists a rational place of $F$ which splits completely into $q+1+t$ rational places of $E$ from \cite[Theorem 3.1.12 and Theorem 3.7.2]{St09}. 

In particular, we assume that $t$ satisfies one of the three conditions given in Lemma \ref{lem:2.5}. 
Hence, there is a cyclic elliptic function field $E/\F_q$ with $q+1+t$ rational places, i.e., there exists a rational place $P$ such that the least positive integer $k$ satisfying $[k]P=O$ is $q+1+t$. 
Hence, the translation group $T_E$ of $E$ is a cyclic group of order $q + 1 + t$ with generator $\tau_P$. 
All rational places of $E$ inherits the cyclic group structure of  $T_E=\langle \tau_P\rangle$ as follows. 
Let $P_j=\tau_P^j(O)=[j]P$ be the rational place of $E$ for each $0\le j\le q+t$. 
By  \cite[Theorem 3.7.1]{St09}, the set $\{P_0,P_1,\cdots, P_{q+t}\}$ consists of all rational places of $E$.

\begin{prop}\label{prop:3.1}
Let $E/\F_q$ be a cyclic elliptic function field defined over $\F_q$ with $q+1+t$ rational places. Let  $T_E=\langle \tau_P\rangle$  be the translation group of $E$ and $F=E^{T_E}$. 
Let $d$ be a positive integer and $Q$ be a place of $E$ with degree $d$. If $\gcd(d,q+1+t)=1$, then $\tau_P^j(Q)$ are pairwise distinct places with degree $d$ of $E$ for all $0\le j\le q+t$.
\end{prop}
\begin{proof}
From \cite[Lemma 5.1]{XKD03} or \cite{E51}, for any $j\in \mathbb{N}$, there exists a unique automorphism $\tau_P^j=\tau_{[j]P}$ of $\Aut(E/\F_q)$ such that 
$$\tau_{[j]P}(Q)+O-Q-[d][j]P$$
is a principal divisor of $E$ for any place $Q$ of $E$ with degree $d$. 
If $\tau_{[j]P}(Q)=Q$, then we have $[dj]P-O$ is a principal divisor.  By Lemma \ref{lem:2.2}, we have $(q+1+t)|dj$ for some $j$. 
Since $\gcd(d,q+1+t)=1$, we have $(q+1+t)|j$. Hence, the decomposition group of $Q|Q\cap F$ is the identity automorphism of $E$. 
Thus, $\tau_P^j(Q)$ are pairwise distinct places  with degree $d$ of $E$ for all $0\le j\le q+t$
 by \cite[Corollary 3.7.2]{St09}. 
\end{proof} 

\begin{lemma}\label{lem:3.2}
Let $E/\F_q$ be an elliptic function field defined over $\F_q$ with $q+1+t$ rational places. Let $B_d$ be the number of places with degree $d$ of $E$. Let $\mu(\cdot)$ denote the M$\ddot{o}$bius function.
Then one has $$B_d=\frac{1}{d}\cdot \sum_{r|d} \mu\left(\frac{d}{r}\right) \cdot \left(q^r+1-\sum_{i=0}^{\lfloor r/2\rfloor} (-1)^{r-i}\frac{(r-i-1)!\cdot r}{(r-2i)!\cdot i!}t^{r-2i}q^i\right).$$ 
 In particular, we have $$B_2=\frac{q^2+q-t^2-t}{2} \text{ and } B_3=\frac{q^3-q+t^3-3qt-t}{3} .$$
\end{lemma}
\begin{proof}
By \cite[Theorem 5.1.15]{St09}, the $L$-polynomial of $E$ is given by 
$$L_E(X)=1+[N(E)-(q+1)]X+qX^2=1+tX+qX^2.$$
Assume that $L_E(X)$ factorizes in $\mathbb{C}[X]$ as $$L_E(X)=(1-\alpha_1 X)(1-\alpha_2 X).$$
Thus, we obtain $\alpha_1+\alpha_2=-t$ and $\alpha_1\alpha_2=q$. 
Let $S_r$ be the sum $\alpha_1^r+\alpha_2^r$. 
By \cite[Proposition 5.2.9]{St09}, for all $d\ge 2$, we have 
$$B_d=\frac{1}{d}\cdot \sum_{r|d} \mu\left(\frac{d}{r}\right) (q^r+1-S_r).$$ 
From Waring's Formula \cite[Theorem 1.76]{LN83}, we have
$$S_r=\sum_{i=0}^{\lfloor r/2\rfloor} (-1)^{r-i}\frac{(r-i-1)!\cdot r}{(r-2i)!\cdot i!}t^{r-2i}q^i.$$
This completes the proof. 
\end{proof}

Let $E$ be a cyclic elliptic function field with $q+1+t$ rational places.
Let $Q$ be a place of $E$ with $\deg(Q)=d\ge 2$ such that $\gcd(d,q+1+t)=1$. 
The existence of such a place $Q$ can be guaranteed by Lemma \ref{lem:3.2}. 
Since $\deg(Q)=d\ge 2g(E)-1=1$, the dimension of the Riemann-Roch space $\mL(Q)$ over $\F_q$ is $\dim \mL(Q)=\deg(Q)+1-g(E)=\deg(Q)=d$. 
It is clear that $\F_q=\mL(0)\subseteq \mL(Q)$. Hence, there exists a vector space $V$ over $\F_q$ of dimension $d-1$ such that $\mL(Q)=\F_q\oplus V$, which is the complementary dual space of $\F_q$ in $\mL(Q)$.	

Let $\Tr$ be the absolute trace map from $\F_{q}$ to $\F_2$ which is defined by $\text{Tr}(\Ga)=\sum_{i=0}^{n-1} \Ga^{2^{i}}=\Ga+\Ga^2+\cdots+\Ga^{2^{n-1}}$ for any $\Ga\in \F_q$. Since $\dim_{\F_q}(V) = d-1$, the set $V \setminus \{0\} $ can be listed as $\{z_1, z_2, \dots, z_{q^{d-1} - 1}\}. $
For any nonzero element $z_i\in V\setminus \{0\}$ with $1\le i\le q^{d-1}-1$, a binary sequence $\bs_i$ associated to $z_i$ can be defined as follows:
$$\bs_i=(s_{i,0},s_{i,1},\cdots,s_{i,q+t}) \text{ with } s_{i,j}=\text{Tr}(z_i(P_j)) \text{ for } 0\le j\le q+t.$$
It is clear that each sequence $\bs_i$ has length $q+1+t$ and the family of binary sequences $\mS=\{\bs_i: 1\le i\le q^{d-1}-1\}$ has size $q^{d-1}-1$. 
We are going to show that this family of binary sequences  $\mS$ has a low correlation and a large linear complexity.

\subsection{Binary sequences with a low correlation via elliptic curves}
In this subsection, we shall show that the family of binary sequences $\mS=\{\bs_i: 1\le i\le q^{d-1}-1\}$ constructed as above has a low correlation. 

\begin{prop} \label{prop:3.3}
For any elements $z_1,z_2\in V\setminus \{0\}$, one has
\begin{itemize}
       \item[(1)] $(z_1+\tau(z_2))_\infty=Q+\tau(Q)$ for any automorphism $\tau\in \Gal(E/F)\setminus \{\rm id\}$; 	
       \item[(2)] $(z_1 + z_2)_\infty = Q$, provided that $z_1\neq z_2$. 
\end{itemize}
\end{prop}
\begin{proof}
For any elements  $z_1,z_2\in V\setminus \{0\}$, the pole divisors of $z_1$ and $z_2$ are given by $(z_1)_\infty=(z_2)_\infty=Q.$
By Lemma \ref{lem:2.1}, we have $\nu_{\tau(Q)}(\tau(z_2))=\nu_Q(z_2)=-1$ and $(\tau(z_2))_\infty=\tau(Q).$
It follows that $z_1+\tau(z_2)\in \mL(Q+\tau(Q)).$ 
For any automorphism $\tau\in \Gal(E/F)\setminus \{\rm id\}$, we have $\tau(Q)\neq Q$ by Proposition \ref{prop:3.1}. 
Hence, we have $$(z_1+\tau(z_2))_\infty=Q+\tau(Q)$$ from the Strict Triangle Inequality \cite[Lemma 1.1.11]{St09}.
		
Suppose that $z_ 1+ z_2 \in \F_q$. It follows that $z_1 + z_2 \in  V \cap \F_q=\{0\}$, i.e., $z_1 =   z_2$ which is a contradiction. Thus, we obtain $(z_1 + z_2 )_ \infty = Q$. 
\end{proof}

\begin{theorem}\label{thm:3.4}
Let $q=2^n$ be a prime power of $2$ for a positive integer $n$. Let $E/\F_q$ be a cyclic elliptic function field defined over $\F_q$ with $q+1+t$ rational places.
Let $F$ be the fixed subfield of $E$ with respect to its translation group $T_E$ generated by $\tau_P$. 
Let $Q$ be a place of $E$ with $\deg(Q)=d\ge 2$ such that $\gcd(d,q+1+t)=1$.
Then there exists a family of binary sequences $\mS=\{\bs_i: 1\le i\le q^{d-1}-1\}$ with length $q+1+t$, size $q^{d-1}-1$ and correlation upper bounded by $$\operatorname{Cor}(\mS)\le (2d+1)\cdot \lfloor 2\sqrt{q}\rfloor+|t|.$$
\end{theorem}
\begin{proof}
                 Let $\sigma=\tau_P$ be the generator of the translation group $T_E$. 
		We begin by computing the autocorrelation of the family of binary sequences $\mS=\{\bs_i: 1\le i\le q^{d-1}-1\}$. 
		For each $1\le i\le q^{d-1}-1$, the autocorrelation of $\bs_i$ at delay $u$ with $1\le u\le q+t$ is given by
		\begin{align*}
		A_u(\bs_i)&=  \sum_{j=0}^{q+t} (-1)^{s_{i,j}+s_{i,j+u}} = \sum_{j=0}^{q+t}(-1)^{\Tr(z_i(P_j))+\Tr(z_i(P_{j+u}))} \\ 
		&= \sum_{j=0}^{q+t} (-1)^{\Tr(z_i(P_j))+\Tr((\sigma^{-u}(z_i))(P_j))}=\sum_{j=0}^{q+t} (-1)^{\Tr((z_i+\sigma^{-u}(z_i))(P_j))}.
		\end{align*}
		The third equality follows from Lemma \ref{lem:2.1}.
		Since $z_i\in V\setminus \{0\}$ and $1\le u\le q+t$, we have $(z_i+\sigma^{-u}(z_i))_\infty=Q+\sigma^{-u}(Q)$ by Proposition \ref{prop:3.3}.
		Consider the Artin-Schreier curve $E_i=E(y)$ defined by $$y^2+y=z_i+\sigma^{-u}(z_i).$$
		By Proposition \ref{prop:2.6}, the place $Q$ and $\Gs^{-u}(Q)$ are the only ramified places in $E_i/E$ and their different exponents are two.
		From the Hurwitz genus formula, we have $$2g(E_i)-2=2[2g(E)-2]+2\deg(Q)+2\deg(\Gs^{-u}(Q)).$$
		Hence, the genus of $E_i$ is $g(E_i)=2\deg(Q)+1=2d+1$ for each $1\le i\le q+t$.
		Let $N_0$ denote the cardinality of the set $S_0=\{0\le j\le q+t: \Tr((z_i+\sigma^{-t}(z_i))(P_{j}))=0\}$.
		Let $N_1$ denote the cardinality of the set $S_1=\{0\le j\le q+t: \Tr((z_i+\sigma^{-t}(z_i))(P_{j}))=1\}$.
		It is clear that $$N_0+N_1=q+1+t.$$
		If $j\in S_0$, then $P_j$ splits into two rational places in $E_i$ by \cite[Theorem 2.25]{LN83} and \cite[Theorem 3.3.7]{St09}.
		If $j\in S_1$, then $P_j$ is inert in $E_i$ and any place of $E_i$ lying over $P_j$ has degree two.
		Let $N(E_i)$ be the number of rational points of the Artin-Schreier curve $E_i$. Then we have 
		$$N(E_i)=2N_0.$$
		From the Serre bound, we have $|N(E_i)-q-1|\le g(E_i)\cdot  \lfloor 2\sqrt{q}\rfloor=(2d+1)\cdot  \lfloor 2\sqrt{q}\rfloor.$  
		Hence, we obtain $$|A_u(\bs_i)|=|N_0-N_1|=|2N_0-q-1-t|=|N(E_i)-q-1-t|\le (2d+1)\cdot  \lfloor 2\sqrt{q}\rfloor +|t|.$$
		
		Now let us consider the cross-correlation of the family of sequences. 
		For two distinct sequences $\bs_i$ and $\bs_j$ in $\mS$ for $1\le i\neq j\le q^{d-1}-1$, the cross-correlation of $\bs_i$ and $\bs_j$ at delay $u$ with $0\le u\le q+t$ is given by
		$$C_u(\bs_i,\bs_j)=\sum_{k=0}^{q+t} (-1)^{\Tr(z_i(P_k))+\Tr(z_j(P_{k+u}))}=\sum_{k=0}^{q+t}(-1)^{\Tr((z_i+\sigma^{-u}(z_j))(P_{k}))}.$$
		Consider the Artin-Schreier curve $E_{i,j}=E(y)$ defined by $$y^2+y=z_i+\sigma^{-u}(z_j).$$
		If $1\le u\le q+t$, then we have $(z_i+\sigma^{-u}(z_j))_\infty=Q+\sigma^{-u}(Q)$ by Proposition \ref{prop:3.3}.
		From the Hurwitz genus formula and Proposition \ref{prop:2.6}, we have $$2g(E_{i,j})-2=2[2g(E)-2]+2\deg(Q)+2\deg(\Gs^{-u}(Q)).$$
		Hence, the genus of $E_{i,j}$ is $g(E_{i,j})=2\deg(Q)+1=2d+1$ for each $1\le i\le q-1$.
		If $u=0$, then we have $(z_i+\sigma^{-u}(z_j))_\infty=Q$ by Proposition \ref{prop:3.3} and the genus of $E_{i,j}$ is $g(E_{i,j})=\deg(Q)+1=d+1$ by Proposition \ref{prop:2.6}.
		Let $N_0^\prime$ denote the cardinality of the set $\{0\le k\le q+t: \Tr((z_i+\sigma^{-u}(z_j))(P_k))=0\}$.
		Let $N_1^\prime$ denote the cardinality of the set $\{0\le k\le q+t: \Tr((z_i+\sigma^{-u}(z_j))(P_k))=1\}$.
		It is clear that $$N_0^\prime+N_1^\prime=q+1+t.$$
		Let $N(E_{i,j})$ be the number of rational points on the Artin-Schreier curve $E_{i,j}$. Similarly, we have
		$$N(E_{i,j})=2N_0^\prime.$$
		From the Serre bound, we have
		$$|C_u(\bs_i,\bs_j)|=|N_0^\prime-N_1^\prime|=|2N_0^\prime -q-1-t|=|N(E_{i,j})-q-1-t|\le (2d+1)\cdot \lfloor 2\sqrt{q}\rfloor+|t|.$$
		The proof is completed by taking the maximum of $|A_u (\mathbf s_j)|$ and $|C_u(\mathbf s_j, \mathbf s_k)|$ among all possible $j,k$. 
\end{proof}

\begin{cor}\label{cor:3.5}
Let $q=2^n$ be a prime power for a positive integer $n$. 
Let $t$ be $0$, or $\pm \sqrt{q}$ if $n$ is even, or $\pm \sqrt{2q}$ if $n$ is odd.
Then there exists a family of binary sequences with length $q+1+t$, size $q-1$ and correlation upper bounded by
$$\operatorname{Cor}(\mS)\le 5\cdot \lfloor 2\sqrt{q}\rfloor+|t|.$$
\end{cor}
	
\begin{proof}
		Let $E$ be a cyclic elliptic function field with $q+1+t$ rational places. 
		By Lemma \ref{lem:3.2}, there exists a place $Q$ of $E$ of degree $\deg(Q)=2$. Note that $\gcd(2,q+1+t)=1$. Thus, 
		this corollary follows from Proposition \ref{prop:3.1},  Proposition \ref{prop:3.3} and Theorem \ref{thm:3.4}. 
\end{proof}

\begin{cor}\label{cor:3.6}
Let $q=2^n$ be a prime power for a positive integer $n$. 
Let $t$ be an integer satisfying one of three items in Lemma \ref{lem:2.5} and $\gcd(3, (-1)^n+1+t)=1$. 
Then there exists a family of binary sequences with length $q+1+t$, size $q^2-1$ and correlation upper bounded by 
$$\operatorname{Cor}(\mS)\le 7\cdot \lfloor 2\sqrt{q}\rfloor+|t|.$$
\end{cor}
	
\begin{proof}
		Let $E$ be a cyclic elliptic function field with $q+1+t$ rational places. 
		By Lemma \ref{lem:3.2} there exists a place $Q$ of $E$ of degree $\deg(Q)=3$ with $\gcd(3,q+1+t)=1$. 
		Then this corollary follows from Proposition \ref{prop:3.1}, Proposition \ref{prop:3.3} and Theorem \ref{thm:3.4}. 
\end{proof}

\subsection{Binary sequences with a high linear complexity}
	In this subsection, we shall show that the family of binary sequences constructed above has a high linear complexity similarly as \cite[Theorem 3.2]{XKD03}. 
	
\begin{theorem}\label{thm:3.7}
	The linear complexity  of the family of binary sequence $\mS=\{\bs_i: 1\le i\le q^{d-1}-1\}$ satisfies 
		\[LC(\mS)\ge \frac{q+1+2t-2(d+1)\sqrt{q}}{2d\sqrt{q}}.\]
\end{theorem}

\begin{proof}
It remains to prove that, for every $1 \le i \le q^{d-1} - 1$, the linear complexity $ \ell (\mathbf s_i)$ of binary sequence $\mathbf s_i$ in $\mathcal{S}$ satisfies 
		\[\ell(\mathbf s_i)\ge \frac{q+1+2t-2(d+1)\sqrt{q}}{2d\sqrt{q}}.\]
		Let $\ell = \ell (\mathbf s_i)$. Then there exist $\ell+1$ binary elements $\lambda_0=1,,\cdots,\lambda_\ell=1\in \F_2$ such that 
		$$\sum_{j=0}^\ell \lambda_j \cdot \Tr(z_i(P_{j+u}))=0$$
		for all integers $u\ge 0$.  Let $\sigma=\tau_P$ be the generator of translation group $T_E$. 
		It is clear that 
		\begin{align*}
			\sum_{j=0}^\ell \lambda_j \cdot \Tr(z_i(P_{j+u}))
			&=\Tr\left(\sum_{j=0}^\ell \lambda_j z_i(P_{j+u})\right)\\
			&=\Tr\left(\sum_{j=0}^\ell \lambda_j z_i(\Gs^{j}(P_{u}))\right)\\
			&=\Tr\left(\sum_{j=0}^\ell \lambda_j \Gs^{-j}(z_i)(P_{u})\right)\\
			&=\Tr(w_i(P_{u})),
		\end{align*}
		provided that $w_i=\sum_{j=0}^\ell \lambda_j \Gs^{-j}(z_i)$. Hence, we obtain
		$$(\Tr(w_i(P_{1})), \Tr(w_i(P_{2})), \cdots,\Tr(w_i(P_{q+1+t})))=\bo.$$  
		Let $L_i=E(y)$ be the Artin-Schreier extension over $E$ generated by 
		$y^2+y=w_i$. From the Strict Triangle Inequality \cite[Lemma 1.1.11]{St09}, the principal divisor of $w_i$ is given by $(w_i)_\infty=\sum_{j=0}^\ell \lambda_j\Gs^{-j}(Q)$. 
		By Proposition \ref{prop:2.6}, we have
		$$2g(L_i)-2=2\cdot (2g(E)-2)+\deg(\operatorname{Diff}(L_i/E))\le 2(\ell+1)\deg(Q).$$
		Hence, the genus of $L_i$ is upper bounded by 
		$$g(L_i)\le d(\ell+1)+1.$$
		It follows that 
		$$2(q+1+t)\le N(L_i)\le q+1+[d(\ell+1)+1]\cdot \lfloor 2\sqrt{q}\rfloor.$$
		Hence, we obtain
		$$\ell\ge \frac{q+1+2t-(d+1)\cdot \lfloor 2\sqrt{q}\rfloor}{d\cdot \lfloor 2\sqrt{q}\rfloor}. \qedhere$$
\end{proof}

\subsection{Comparison} 
In \cite{XKD03}, Xing {\it et al.} provided a construction of binary sequences with a low correlation and a large linear complexity via elliptic curves. 
The family of binary sequences given in \cite{XKD03} is different from our family of binary sequences given in this manuscript, although both families of binary sequences  share the same length, the same upper bound of correlation and the same lower bound of linear complexity from Theorem \ref{thm:3.4}, Theorem \ref{thm:3.7} and \cite[Theorem 5.5]{XKD03}.

Let $d$ be a positive divisor which is relatively prime with $q+1+t$. 
There are exactly $r=B_d/(q+1+t)$ places of $F=E^{T_E}$ of degree $d$ which split completely in the extension $E/F$. 
Let $Q_1,Q_2,\cdots,Q_r$ be the different places of $E$ with degree $d$ such that their restrictions are pairwise distinct. Xing {\it et al.}  chosen a unique element in every set $\mL(Q_i)\setminus \F_q$ for $1\le i\le r$.  Hence, the size of their family of binary sequences is $r=B_d/(q+1+t)$. 
However, we choose all nonzero elements in the complementary dual linear space $V$ of $\F_q$ in the Riemann-Roch space $\mL(Q)=\F_q\oplus V$ for any fixed place of $Q$ of $E$ with degree $d$ in this manuscript. By Lemma \ref{lem:3.2}, Corollary \ref{cor:3.5} and Corollary \ref{cor:3.6}, the size of the new family of binary sequences in this manuscript is larger than the one given in \cite{XKD03}.

\section{An explicit construction of such a family of binary sequences}\label{sec:4}
The previous section provides a theoretical framework for constructing a new family of binary sequences with a low correlation and a high linear complexity using elliptic function fields. 
In this section, we will make this construction explicitly. 
We achieve this by determining the automorphisms and rational places explicitly. Additionally, we present some numerical results generated by using the software Sage. 
It turns out that the actual correlation of our family of binary sequences is significantly lower than the estimated upper bounds provided in Theorem \ref{thm:3.4}.

\subsection{Algorithm}\label{sec:4.1}
Now we provide an algorithm to generate the family of binary sequences with a low correlation and a high linear complexity constructed in Section \ref{sec:3}.
Such an explicit construction of sequences with a low correlation can be given as follows.

%\fbox{\begin{minipage}{35em}
\bigskip
\begin{center}
{\bf Algorithm: Explicit construction of sequences with a low correlation}
\end{center}
\begin{itemize}
\item {\bf Step 1:} Input a prime power $q=2^n$ and a fixed integer $t$ given in Lemma \ref{lem:2.5}.
\item {\bf Step 2:} Find a cyclic elliptic function field $E/\F_q$ with exactly $q+1+t$ rational places.
\item {\bf Step 3:} Find a rational place $P=(\alpha_1,\beta_1)$ of $E$ with order $q+1+t$ and calculate $P_j=[j]P=(\alpha_j,\beta_j)$ for $0\le j\le q+t$ from \cite[Group Law Algorithm 2.3]{Si86}.
\item {\bf Step 4:} Let $\lambda=\frac{y-\beta_1}{x-\alpha_1}$ and $v=\frac{\beta_1x-\alpha_1y}{x-\alpha_1}$. 
The automorphism $\tau_P$ of $E$ can be determined explicitly as 
$$\tau_P(x)=\lambda^2+a_1\lambda-a_2-x-\alpha_1, \quad \tau_P(y)=-(\lambda+a_1)\tau_P(x)-v-a_3. $$
Let $T_E$ be the translation group $E$ which is generated by $\tau_P$. Let $F$ be the fixed subfield of $E$ with respect to $T_E$. 
\item {\bf Step 5:} Find the Riemann-Roch space of $\mL(Q)$ for any place $Q$ with degree $d\ge 2$ which is relatively prime to $q+1+t$ and determine the vector space $V$ such that $\mL(Q)=\F_q\oplus V$. 
\item {\bf Step 6:} Label all elements of $V\setminus \{0\}$ with $\{z_i: 1\le i\le q^{d-1}-1\}$. 
\item {\bf Step 7:} Output a family of sequences $\mS=\{\bs_i: 1\le i\le q^{d-1}-1\}$ defined by $$\bs_i=(s_{i,0},s_{i,1},\cdots,s_{i,q+t}) \text{ with } s_{i,j}={\Tr(z_i(P_j))} \text{ for } 0\le j\le q+t.$$
\item {\bf Step 8:}  Output the correlation $\operatorname{Cor}(\mS)=\max\{\{|\sum_{k=0}^{q+t} (-1)^{s_{i,k}}(-1)^{s_{i,k+u}}|:1\le i\le q^{d-1}-1, 1\le u\le q+t\} \cup \{|\sum_{k=0}^{q+t} (-1)^{s_{i,k}}(-1)^{s_{j,k+u}}|: 1\le i\neq j\le q^{d-1}-1, 0\le u\le q+t\}\}.$
\end{itemize}
\bigskip

\begin{rem}
Now we estimate the  computational complexity on generating the family of binary sequences from our construction. We consider the case when $d=2$. 
In Step 2, to find an  elliptic curve with a specific number of points, one would typically generate random curves and count their points. One of the most effective algorithms for counting points on elliptic curves over finite fields is the Schoof-Elkies-Atkin (SEA) algorithm  which takes  $O((\log q)^6)$ time \cite{V20}. Thus the worst case complexity for Step 1 is 
 $O(q^5(\log q)^6)$.
In Step 3, for each element, we need to check if it is a generator which takes at most $O(q)$, thus the total time is $O(q^2(\log q)^2)$.
In Step 7, we do the evaluation which takes $O(q^2(\log q)^3)$.
In Step 8, we need to compute the correlation between any two sequences as well as the cyclic shift, thus  the computational complexity is $O(q^4)$. 
The complexity of other steps are small, so they can be neglected. Thus, the overall complexity of generating the family of binary sequences is $O(q^5(\log q)^6)$. 
\end{rem}

\subsection{Numerical results}\label{sec:4.2}
In this subsection, we list some numerical results on the parameters of the family of binary sequences obtained from elliptic curves with the help of Sage software. We consider the case where $d=2$ in the following.

\begin{table}[h]
	\setlength{\abovecaptionskip}{0pt}%
	\setlength{\belowcaptionskip}{10pt}%
	\caption{PARAMETERS OF OUR SEQUENCES WITH LENGTH EQUAL TO $q$}
	\center
	\begin{tabular}{@{}|c|c|c|c|c|@{}}
		\toprule
		Field Size & Seq. Length & Family Size & Max Correlation \\ \midrule
		64    & 64            & 63         & 38                   \\ \midrule
		128   & 128           & 127        & 60                        \\ \midrule
		256  & 256            & 255        & 86                      \\ \midrule{}
		512   & 512            & 511         &  124                     \\\midrule
		1024    & 1024            & 31       &  184                   \\ \midrule
		2048    & 2048            & 63         & 276                   \\ \midrule
		4096  & 4096           & 127        & 416                      \\ 
	
		\bottomrule
	\end{tabular}
\label{tab:2}
\end{table}

Let $q=2^n$ be a prime power. 
By Lemma \ref{lem:2.5}, there exists a cyclic elliptic function field $E / {\F}_q$ with exactly $q$ rational places. 
By Corollary \ref{cor:3.6}, we can construct a family of binary sequences with length $q$, family size $q^2-1$ and correlation upper bounded by $14\sqrt{q}+1$.  
We list some numerical results for this case in Table \ref{tab:2}.  
It turns out that the real correlation of our family of binary sequence of length $q$ is much better than the estimated  upper bounds given in Theorem \ref{thm:3.4}.  

Let $q=2^n$ be a prime power.  Let $t$ be $\sqrt{q}$ if $n$ is even, or $\sqrt{2q}$ if $n$ is odd. It is trivial that $\gcd(2, q+1+t)=1$. 
By Corollary \ref{cor:3.5}, we can construct a family of binary sequences with family size $q-1$ and odd length $q+1+t$.  We list some numerical results for this case in Table \ref{tab:3}.

\begin{table}[h]
	\setlength{\abovecaptionskip}{0pt}%
	\setlength{\belowcaptionskip}{10pt}%
	\caption{PARAMETERS OF OUR SEQUENCES OF ODD LENGTHS}
	\center
	\begin{tabular}{@{}|c|c|c|c|@{}}
		\toprule
		Field Size & Seq. Length & Family Size & Max Correlation  \\ \midrule
		 64 &   73    &63   & 39 \\ \midrule
	    128  & 145  &127     & 57     \\ \midrule
		256&   273    &255  & 89 \\ \midrule
		512&   545   & 511     &    137   \\ \midrule
		1024&1057     &1023   & 191\\
		\bottomrule
	\end{tabular}
\label{tab:3}
\end{table}

In fact, we are able to generate more families of binary sequences with flexible lengths $q+1+t$ with $t$ being an integer between $-2\sqrt{q}$ and $2\sqrt{q}$.

\iffalse
Let $q=2^n$ be a prime power.  Let $t$ be an odd integer between $-2\sqrt{q}$ and $2\sqrt{q}$ with 
$\gcd(2, q+1+t)=1$.  
From Corollary \ref{thm:3.4}, we can construct a family binary sequences with family size $q^2-1$ and various even lengths $q+1+t$. We list some corresponding numerical results in Tables \ref{tab:4} for different values of $q$.

\begin{table}[h]
	\setlength{\abovecaptionskip}{0pt}%
	\setlength{\belowcaptionskip}{10pt}%
	\caption{PARAMETERS OF OUR SEQUENCES OF EVEN LENGTHS}
	\center
	\begin{tabular}{@{}|c|c|c|c|c|c|c|c|c|@{}}
		\toprule
		\iffalse
		q=64& Seq. Length      & 52          & 56       & 60  & 68 & 72 &  76 & 80         \\ 
		& Max Correlation  & 36          & 40       & 36  & 44  & 40 & 48&40 \\ \midrule
		\fi
	q=128& Seq. Length      & 112        & 120       & 124 & 132 & 136 &  144   &148        \\ 
		& Max Correlation  & 56          &60     & 54  & 60 & 64 & 62&66\\ \midrule
	q=256 &Seq. Length      & 232       & 240      & 248 & 264 & 272 &  276  &284      \\ 
		& Max Correlation  & 84          &88     & 84  & 86& 92& 90&92\\ \midrule
	q=512 &	Seq. Length      & 480        & 496      & 508 & 520 & 528 &  540   &556        \\
		& Max Correlation  & 132         &132    & 128 & 128 & 138 & 132&140\\   \midrule
			q=1024& Seq. Length  &   964   &  988      &  996  &  1036 &  1056 &  1072  &   1084        \\ 
		& Max Correlation  &    204      &  196  & 202 &  190  & 196  & 212 & 192 \\ 
		\bottomrule
	\end{tabular}
\label{tab:4}
\end{table}
\fi

\section{Conclusion}
While significant progress has been made in constructing binary sequences with good parameters, there remains a shortage of options for sequence lengths in existing results. Therefore, it is imperative to develop sequences with alternative lengths to accommodate diverse scenarios.
Most existing families of binary sequences  use the cyclic structures inherent in the multiplicative cyclic group of $\mathbb{F}_{p^n}$. Notably, recent work by \cite{HJMX24,JMX22} capitalized on the cyclic structure of $q+1$ rational points on projective lines over $\mathbb{F}_q$ to design a family of binary sequences with a length of $q+1$.

In this paper, we exploit the cyclic group structure of rational points of elliptic curves to design a family of binary sequences of length $2^n+1+t$ with a low correlation and a high linear complexity. 
Our family size surpasses that of \cite{XKD03}. Furthermore, our construction offers the potential to derive additional sequences with a low correlation and other desirable parameters.

\end{document}